\documentclass[12pt,a4paper,oneside]{article}
\usepackage{amsmath,amstext,amssymb,amsopn,amsthm,color}
\theoremstyle{plain}
\newtheorem{thm}{Theorem}[section]

\newtheorem{lem}[thm]{Lemma}
\newtheorem{prop}[thm]{Proposition}
\newtheorem{cor}[thm]{Corollary}

\theoremstyle{definition}

\theoremstyle{remark}
\newtheorem*{rem*}{Remark}

\newcommand{\R}{\mathbb{R}}

\newcommand{\Z}{\mathbb{Z}}
\newcommand{\C}{\mathbb{C}}
\renewcommand{\H}{\mathbb{H}}

\renewcommand{\leq}{\leqslant}
\renewcommand{\geq}{\geqslant}

\newcommand{\pref}[1]{(\ref{#1})}

\def\o{\over}
\def\({\left(}
\def\){\right)}
\def\[{\left[}
\def\]{\right]}
\def\<{\langle}
\def\>{\rangle}

\title {Hitting distributions of planar Brownian motion
\footnotetext{2000 MS Classification:
    Primary ;
    Secondary . {\it Key words and phrases}: hitting times and distributions,       
    Poisson kernels,  Green functions,  Cauchy stable and 
    relativistic processes, hyperbolic spaces.  
    }}
\author{ T. Byczkowski, J. Ma\l{}ecki and M. Ryznar\\ Institute of Mathematics
and Computer Sciences\\ Wroc\l{}aw
    University of Technology, Poland}
\date{}

\begin{document}
\maketitle

\begin{abstract}
The purpose of the paper is to find the joint distribution of the hitting time and place of two-dimensional Brownian motion hitting the negative horizontal axis. We provide various formulas for Green functions as well as for the conditional 
distributions. In the end we treat an analogous problem for the case when vertical component has a negative drift.
\end{abstract}

\newpage
\section{Introduction}
The aim of the paper is to investigate two-dimensional Brownian motion hitting the negative part of the horizontal axis. We compute  its various hitting characteristics.
Although there is a close connection with hitting distributions for Cauchy L\'evy 
process as well as with relativistic Cauchy process (see \cite{BMR}), we are mainly interested in the two-dimensional (Brownian) situation. We show that most characteristics can be recovered from the correspondig ones for  two-dimensional Cauchy process by an appropriate application of Fourier transform. 
\section{Preliminaries}

\subsection{Bessel functions}
Various potential-theoretic objects in our paper are expressed
 in terms of modified Bessel functions $K_{\vartheta}$ of the third kind, also called Macdonald
 functions. For convenience
 we collect here basic information about these functions.

The modified Bessel function $I_{\vartheta}$ of the first kind is defined by (see, e.g. \cite{E1}, 7.2.2 (12)):
  \begin{equation}
    \label{I_definition}
     I_\vartheta(z) = \(\frac{z}{2}\)^\vartheta \sum_{k=0}^{\infty} \left(\frac{z}{2}\right)^{2k}\frac{1}{k!\Gamma(k+\vartheta+1)} \/, \quad    z \in \C \setminus (-\R_{+}) \/,
  \end{equation}
  where $\vartheta\in \R$ is not an integer and $z^{\vartheta}$ is the branch that is analytic on $\C \setminus (-\R_{+})$  and positive on
  $\R_{+} \setminus \{0\}$
  
  The modified Bessel function of the third kind is defined by (see \cite{E1}, 7.2.2 (13) and (36)):
  \begin{eqnarray}
    \label{K_definition1}
    K_\vartheta (z) &=& \frac{\pi}{2\sin(\vartheta\pi)}\left[I_{-\vartheta}(z)-I_\vartheta(z)\right]\/,\quad
    \vartheta \notin \Z \/,\\
    \label{K_definition2}
    K_n (z) &=& \lim_{\vartheta\to n}K_\vartheta(z) = (-1)^n
    \frac{1}{2}\left[\frac{\partial I_{-\vartheta}}{\partial \vartheta} -
    \frac{\partial I_{\vartheta}}{\partial \vartheta}\right]_{\vartheta=n}\/,\quad n\in \Z\/.
  \end{eqnarray}
  For $\vartheta=n+1/2$ we have the following expression for the function $K_{n+1/2}$
  \begin{equation}
     \label{Kn12}
     K_{n+1/2}(z) = \left(\frac{\pi}{2z}\right)^{1/2}e^{-z}\sum_{k=0}^n \frac{(n+k)!}{k!(n-k)!(2z)^k}\/.
  \end{equation}
  We will also use the following integral representations of the function $K_\vartheta(z)$ (\cite{E1}, 7.11 (23) or \cite{GR}, 8.432 (6)):
  \begin{eqnarray}
     \label{Macdonald}
     K_\vartheta(z) = 2^{-\vartheta-1} z^{\vartheta} \int_0^\infty e^{-t}e^{-\frac{z^2}{4t}}t^{-\vartheta-1}\,dt\/,
  \end{eqnarray}
  where $\Re(z^2)>0$, $|\arg z|<\frac{\pi}{2}$. Moreover (see \cite{GR}, 8.432 (3)),
  \begin{equation}
    \label{Macdonald0}
     K_\vartheta(z) = \left(\frac{z}{2}\right)^{\vartheta}
     \frac{\Gamma(1/2)}{\Gamma(\vartheta+1/2)}\int_1^\infty \frac{e^{-zt}}{(t^2-1)^{1/2-\vartheta}}\,dt\/,
  \end{equation}
  where $\Re(\vartheta+1/2)>0$, $|\arg z|<\frac{\pi}{2}$.

\subsection{Relativistic processes}
  We now make a detour from the main topic and point out some connections with the relativistic
 $\alpha$-stable process and its hitting probabilities. We begin with recalling some basic
 facts about these processes.
  
  We begin with recalling the definition of the standard $\alpha/2$-stable subordinator $S_t^{\alpha}$ with the Laplace transform
 $E^0 e^{-\lambda S_t^{\alpha}} = e^{-t \lambda^{\alpha/2}}\,$. Throughout the whole
 paper $\alpha$ denotes the stability index of the process and we always assume $0<\alpha<2$.
  The transition density function of $S_t^{\alpha}$ will be denoted by
   $\theta_t^{\alpha}(u)$. Here $u, t>0$.

 For $m>0$ we define another subordinating process $T_t^{\alpha,m}$ by
 modifying $\theta_t^{\alpha}(u)$ as follows:

  \begin{equation} \label{talpha}
  \theta_t^{\alpha, m}(u)= e^{mt}\,\theta_t^{\alpha}(u)\,e^{-m^{2/\alpha}\,u},  \quad u>0\,.
  \end{equation}
  The Laplace transform of $T_t^{\alpha ,m}$ is:
  \begin{equation*}
  E^{0} e^{-\lambda T_t^{\alpha ,m}} =e^{mt}\, e^{-t (\lambda+m^{2/\alpha})^{\alpha/2}}\,.
  \end{equation*}
  Let $B_t$ be the Brownian motion in $\R^d$ with the characteristic
   function $E^{0}e^{i(\xi,B_t)} = e^{-t|\xi|^2}$.
  The transition density function of $B_t$ is denoted by $g_t$ and is of the form:
  \begin{equation*}
  g_t(u)= \frac{1}{(4\pi t)^{d/2}}\,e^{-|u|^2/4t}\,.
  \end{equation*}

  Assume that the processes $T_t^{\alpha ,m}$ and $B_t$ are stochastically independent.
  Then the process $X_t^{\alpha,m}= B_{T_t^{\alpha ,m}}$ is called
  {\it the $\alpha$-stable relativistic process}
  (with parameter $m$).
   In the sequel we will use the generic notation $X_t^m$ instead of
    $X_t^{\alpha,m}$.
   If $m=1$ we then write
   $T_t^{\alpha}$ instead of $T_t^{\alpha ,1}$ and
   $X_t$ instead of $X_t^1$.

   When $m=0$ we obtain {\it the $\alpha$-stable rotation invariant L\'evy process} which
   is denoted by $Z_t$.

   We obtain
   \begin{eqnarray*}
   P^x(X_t^m \in A) &=& E^x [{\bf{1}}_A(B_{T_t^{\alpha ,m}})]\\
   = \int_0^{\infty} \left[\int_A g_u(x-y)\,dy\right]\theta_t^{\alpha, m}(u)\,du &=&
   \int_A  \left[\int_0^{\infty} g_u(x-y)\,\theta_t^{\alpha, m}(u)\,du \right]\,dy\,.\\
   \end{eqnarray*}

  This provides the formula for the transition density function of the process $X_t^m$:

   \begin{equation} \label{reldensity0}
  p^m_t(x)=\int_0^\infty\theta_t^{\alpha,m}(u)\, g_u(x)\, du\,.
  \end{equation}
  $p^m_t(x)\,,t>0$, is a semigroup under convolution.

 A particular case
   when $\alpha=1$ yields
  {\it the relativistic Cauchy semigroup} on $\R^{d}$ with parameter $m$ and is denoted by $\tilde{p}^m_{t}$.
  The formula below exhibits the explicit form of this transition density function.
   \begin{lem}
  \begin{equation} \label{Cauchyrel}
    \tilde{p}^m_{t}(x)=
     2(m/2\pi)^{(d+1)/2}\, te^{mt}
      \frac{K_{(d+1)/2}(m(|x|^2+t^2)^{1/2})}{ (|x|^2+t^2)^{\frac{d+1}{4}}}\,.
    \end{equation}
   \end{lem}
   \begin{proof}
   Observe that $\theta_t^1(u)$, the transition density function of the $1/2$-stable subordinator, is of the   form
   \begin{equation*}
   \theta_t^1(u) = \frac{t}{\sqrt{4\pi}} u^{-3/2} e^{-t^2/4u}\,,
   \end{equation*}
   so, taking into account \pref{Macdonald}, we obtain
   \begin{eqnarray*}
     \tilde{p}^m_{t}(x)&=&
     e^{mt} \int_0^{\infty} \frac{1}{(4\pi u)^{d/2} } e^{-|x|^2/4u} e^{-m^2u}
     \frac{t}{\sqrt{4\pi}} u^{-3/2} e^{-t^2/4u}\, du \\
     &=& \frac{t e^{mt}}{(4\pi)^{\frac{d+1}{2}}} \int_0^{\infty} e^{-m^2 u}
     e^{-(|x|^2+t^2)/4u} \frac{du}{u^{\frac{d+1}{2}+1}} \\
     &=&    2(m/2\pi)^{(d+1)/2}\, te^{mt}
   \frac{K_{(d+1)/2}(m(|x|^2+t^2)^{1/2})}{(|x|^2+t^2)^{\frac{d+1}{4}}}\,.
   \end{eqnarray*}
  \end{proof}

 In what follows we will work within the framework of the so-called $\lambda$-{\it potential theory},
 for $\lambda>0$.

  The kernel of the $\lambda$-resolvent of the semigroup generated by $X_t^m$ will be
 denoted by $U^m_{\lambda}(x)$ and will be called the $\lambda$-{\it potential}, for $\lambda>0$. 
 We have
  \begin{equation*}
  U^m_{\lambda}(x) = \int_0^{\infty} e^{-\lambda t}\,p_t^m(x)\,dt\,.
  \end{equation*}

 The function has a particularly
 simple expression when $\lambda=m$ and we state it for further references.%

  \begin{lem} ($m$-potential for relativistic process with parameter $m$)
 \begin{equation}\label{m-potential}
 U^m_m(x)= \frac{2^{1-(d+\alpha)/2}}{\Gamma(\alpha/2)\pi^{d/2}}
 \frac{m^{\frac{d-\alpha}{2 \alpha}}\,
 K_{(d-\alpha)/2}(m^{1/\alpha}|x|)}{  |x|^{(d-\alpha)/2}}\,.
 \end{equation}
 \end{lem}
 \begin{proof}
 We provide  calculations for $m=1$; the general case follows from 
 Observe first that the potential kernel of the $\alpha/2$-stable subordinator 
 is well-known (and easy   to obtain via Laplace transform). Namely, we have
  \begin{equation*}
 \int_0^\infty\theta_t^{\alpha}(u)dt = \frac{u^{\alpha/2-1}}{\Gamma(\alpha/2)} 
 \end{equation*}
  This and \pref{Macdonald} yield
 \begin{eqnarray*}
 U_1(x)&=&\int_0^\infty e^{-t} p_t(x)\,dt=\int_0^\infty
  \int_0^\infty g_u(x)e^{-u}\theta_t^{\alpha}(u)\,du\,dt \nonumber\\
    &=& \int_0^\infty \frac{1}{{(4\pi u)^{d/2}}}
   e^{-\frac {|x|^2}{4u}}e^{-u} 
   \left(\int_0^\infty\theta_t^{\alpha}(u)dt\right)\,du \nonumber\\
  &=&\frac{1}{{\Gamma(\alpha/2)(4\pi)^{d/2}}}\int_0^\infty
   e^{-\frac {|x|^2}{4u}}e^{-u} \frac{du}{u^{\frac{d-\alpha}{2} +1}} \nonumber\\ &=&
   \frac{2^{1-(d+\alpha)/2} }{ {\Gamma(\alpha/2)\pi^{d/2}}}
  \frac{ K_{(d-\alpha)/2}(|x|)}{|x|^{(d-\alpha)/2}}  \,.
  \end{eqnarray*}
  \end{proof}
In what follows we
 denote by $U_1$ the $\lambda$-potential for $\lambda=m=1$.

  In the next lemma we compute the  Fourier transform of the transition
  density function \pref{reldensity0}.
  \begin{lem}
   \label{relcharfunc}
  The Fourier transform of $\alpha$-stable relativistic transition density function $p^m_t$ is of the form:
  \begin{equation*}
 \widehat{p^m_t}(z) =
   e^{mt}e^{-t(|z|^2+m^{2/\alpha})^{\alpha/2}}\,.
  \end{equation*}
  \end{lem}
  \begin{proof}
  \begin{eqnarray*}
  \widehat{p^m_t}(z)&=&\int_{\R^{d}}  p^m_t(x)e^{i(z,x)}\,dx
  = \int_{\R^{d}}\int_0^{\infty} e^{mt} g_u(x)e^{-m^{2/\alpha}u}\theta_t^{\alpha}(u)\,du\,
   e^{i(z,x)}\,dx \\
    &=& e^{mt} \int_0^{\infty}e^{-u|z|^{2}}e^{-m^{2/\alpha}u}\theta_t^{\alpha}(u)\,du
    = e^{mt} \int_0^{\infty}e^{-u(|z|^{2}+m^{2/\alpha})}\theta_t^{\alpha}(u)\,du \\
    & = & e^{mt}e^{-t(|z|^2+m^{2/\alpha})^{\alpha/2}}.
    \end{eqnarray*}
  \end{proof}

  Specializing this to the case $\alpha=1$ we obtain
  \begin{equation}\label{relFourier}
   \widehat{\tilde{p}^m_t}(z)=
  e^{mt}e^{-t(|z|^2+m^{2})^{1/2}}.
     \end{equation}

  From the Fourier transform we obtain the following scaling property:
  \begin{equation}\label{scaling}
  {p}^m_{t}(x)=m^{d/\alpha}{p}^1_{mt}(m^{1/\alpha}x).
    \end{equation}
  In terms of one-dimensional distributions of the relativistic process (starting from
  the point $0$) \pref{scaling} reads as
  \begin{equation*}
  X_t^m \sim m^{-1/\alpha} X_{mt}^1\,,
  \end{equation*}
  where $X_t^m$ denotes the relativistic $\alpha$-stable process with parameter $m$
  and "$\sim$" denotes equality of distributions.
  Because of this scaling property, we often restrict our attention to the case
   when $m=1$, if not specified otherwise.
   When $m=1$ we omit the superscript "$1$", i.e. we write $p_{t}(x)$
   instead of $p_t^1(x)$.

\section{Hitting planar Brownian Motion}

\subsection{Distribution of $B_{\tau_{D}}$.}

In this section we work in the following setting: we consider 
complex Brownian motion $B(t)+iW(t)$ 
hitting the negative part of the
real axis. We are interested in finding the joint distribution of the exit time and place. We assume, as usual, the independence of the processes $W$ and $B$ and, for the sake of a suitable normalization, that
 $EW^2(t)=EB^2(t)=2t$. We further
denote
\begin{equation} \label{setD}
D=((-\infty,0]\times \{0\})^{c} \quad \text{and} \quad
\tau_D=\inf\{t>0; (B(t),W(t))\notin D\}.
\end{equation}

The first observation follows from a
straightforward application of L\'evy theorem (see, e.g. \cite{Ba}) 
of conformal invariance of two-dimensional
Brownian motion, applied to the half-plane. 
\begin{thm} [Poisson kernel of $D$]\label{pkf}
  Let $z<0<x$. The density function of $B_{\tau_D}$
 with respect to $P^{(x,0)}$ is of the form
  \begin{equation}
    \label{h1}
    h(x,z)= {1 \o \pi} \left( x \o -z \right)^{1/2} {1 \o x-z}\,,
  \end{equation}
  and with respect to $P^{(x,y)}$, where $(x,y)\in D$
  \begin{equation}\label{h2}
    h((x,y),z) = {1 \o \pi} \int_0^{\infty} {|y|\o (x-v)^2 + y^2} h(v,z) \,dv +
    {1 \o \pi} {|y|\o (x-z)^2 + y^2}\,,
     \end{equation}
 or, equivalently,
  \begin{equation}\label{h3}
      h((x,y),z) = {2^{-1/2} \o \pi} \left( {|w|+x \o -z}\right)^{1/2}
      {|w|-z \o (x-z)^2 + y^2}\,.
       \end{equation}
 In the last formula we denote $w=(x,y)$.
\end{thm}

\begin{proof}
 We begin with justifying the last formula in Theorem \ref{pkf}.
 We denote by $\H=\{z; \Re{z}>0\}$ the right
 half-plane and by $\tau_{\H}$ the first exit time from $\H$ of the
 Brownian motion $Z_t$ starting from the point $u=(u_1,u_2) \in \H$.
 Let $Y=Z^2$. By Paul L\'evy conformal invariance of Brownian motion theorem the process
 $Y$ is the Brownian motion with changed time. Moreover, the half-plane $\H$ is transformed
 onto the set $D=((-\infty,0]\times \{0\})^{c}$ and the hitting place of $D^{c}$ by
 the process $Y$ is transformed under the mapping $w \to w^{1/2}$
 into the corresponding hitting place  of $\H^{c}$ by the process $Z$. In the calculations below
 the starting point of the process $Y$ is denoted by $w=x+iy$. We then take
 $\Re{w^{1/2}}=\sqrt{(|w|+x)/2}$ and $\Im{w^{1/2}}=\text{sgn}(y) \sqrt{(|w|-x)/2}$.
 Thus, we obtain
 \begin{eqnarray*}
 P^w(Y_{\tau_D} \in (-u,0)\times \{0\}) &=&
 P^{w^{1/2}}(Z_{\tau_{\H}} \in \{0\}\times (-u^{1/2},u^{1/2})) \\
 &=& \int_{-u^{1/2}}^{u^{1/2}} {1 \o \pi} {\Re{w^{1/2}} \o |iv-w^{1/2}|^2}\, dv= \\
 \int_{-u^{1/2}}^0 {1 \o \pi} {\Re{w^{1/2}} dv \o (v-\sqrt{(|w|-x)/2})^2 +(|w|+x)/2} &+&
 \int_{0}^{u^{1/2}} {1 \o \pi} {\Re{w^{1/2}} dv \o (v-\sqrt{(|w|-x)/2})^2 +(|w|+x)/2} =\\
 \int_{0}^{u} {1 \o 2 \pi} {\Re{w^{1/2}} \o z^{1/2}} {dz \o z+z^{1/2}\sqrt{2(|w|-x)}+|w|}&+&
 \int_{0}^{u} {1 \o 2 \pi} {\Re{w^{1/2}} \o z^{1/2}} {dz \o z-z^{1/2}\sqrt{2(|w|-x)}+|w|}= \\
 &=& \int_{-u}^{0} {1 \o \sqrt{2} \pi} {\sqrt{|w|+x} \o \sqrt{-z}}
 {|w|-z \o z^2-2zx+|w|^2} \,dz\,.
   \end{eqnarray*}
  We just have obtained the last formula.
 Now, if we put $y=0$ in the last formula then $x>0$, ($(x,y) \in D$), and we obtain the
 first formula. To prove the remaining part, observe that for $(x,y) \in D$ and for a
 bounded Borel function $f$ we obtain for $\tau:=\tau_{\H}$
 \begin{eqnarray*}
 E^{(x,y)}f(B_{\tau_D}) &=& E^{(x,y)}[\tau<\tau_D; f(B_{\tau_D})] +
 E^{(x,y)}[\tau=\tau_D; f(B_{\tau_D})] \\
  &=& E^{(x,y)}[B_{\tau}>0;E^{(0,B_{\tau})} [f(B_{\tau_D})]] +
   E^{(x,y)}[B_{\tau}\leq 0; f(B_{\tau})] \\
  &=& {1 \o \pi} \int_0^\infty {|y| \o (x-v)^2+y^2} \left\{\int_{-\infty}^0 h(v,z)f(z) dz\right\} dv
  + {1 \o \pi} \int_{-\infty}^0 {|y| \o (x-z)^2+y^2} f(z) dz\,.
   \end{eqnarray*}

  The proof is complete.
\end{proof}
We state now one more formula. Namely, if we
 substitute in  \pref{h3} $w=(0,y)$ we obtain
 \begin{cor}
 \begin{equation}\label{h4}
      h((0,y),z) = {2^{-1/2} \o \pi} \left( {|y| \o -z}\right)^{1/2}
      {|y|-z \o y^2 + z^2}\,.
       \end{equation}
 \end{cor}



Although
the form of $d$-dimensional $\alpha$-stable Poisson kernel for the half-space
is well-known, for all $\alpha$, $0<\alpha <1$, we provide
an alternative proof. The simplest case is when $\alpha=1$ and we then apply once again L\'evy's theorem.
Before formulating the theorem, we introduce some notation. As before, by $\H$ we denote
the (right) half-space; this time $d$-dimensional, $d\geq 1$:
$\H=\{(x_1,\text{\bf{x}}); \text{\bf{x}} \in \R^{d-1}, \,x_1>0\}$.

\begin{thm}[d-dimensional $\alpha$-stable Poisson kernel]\label{dpkf}
The $\alpha$-stable Poisson kernel of the set $\H$ is of the form:
 \begin{equation} \label{dp}
 P(x,u)=C_{\alpha}^d \left({x_1 \o -u_1} \right)^{\alpha/2} {du \o |x-u|^d},\quad
 u_1<0<x_1,
 \end{equation}
 where $C_{\alpha}^d = \Gamma(d/2) \sin(\pi \alpha/2)/ \pi^{1+d/2} $.
 \end{thm}
 \begin{proof} We assume in what
 follows that $d>1$.
 First of all, observe that, according to general facts from potential theory,
 it is enough to show that for all $y_1<0$ the following identity holds:
 \begin{equation} \label{wym}
 \int_{u_1 <0} C_{\alpha}^d
 \left({x_1 \o -u_1} \right)^{\alpha/2} {1 \o |x-u|^d}
 {du \o |u-y|^{d-\alpha}} = {1 \o |x-y|^{d-\alpha}} \,.
 \end{equation}
 We prove this identity by taking
 the $(d-1)$-dimensional Fourier transform with respect to $\text{\bf{y}}\in \R^{d-1}$.

 Observe that taking into account the form of the $(d-1)$-dimensional Fourier
 transform of symmetric Cauchy distribution in $\R^{d-1}$ we obtain
 \begin{equation} \label{Cauchy}
 \int_{\R^{d-1}} {\Gamma(d/2) \o \pi^{d/2}}
 {(x_1-u_1) e^{i(\text{\bf{v}},\text{\bf{z}})} \o (|\text{\bf{v}}|^2 + (x_1-u_1)^2)^{d/2}}\
 d\text{\bf{v}} = e^{-(x_1-u_1)|\text{\bf{z}}|} \  .
 \end{equation}
 At the same time, by the following simple formula
 \begin{equation}
 \int_0^{\infty} e^{-y^2/4t} e^{-\lambda^2/4t} {dt \o t^{1+s}} =
 {2^{2s} \Gamma(s) \o (y^2+\lambda^2)^s} \,, \quad s>-1\,,
 \end{equation}
 we obtain

 \begin{eqnarray} \label{pot1}
 &{}&\int_{\R^{d-1}}
  { e^{i(\text{\bf{w}},\text{\bf{z}})} \o (|\text{\bf{w}}|^2 + (u_1-y_1)^2)^{d-\alpha \o 2}}\,
  d\text{\bf{w}} =
  \int_0^{\infty} \int_{\R^{d-1}}
  {2^{\alpha-1} \pi^{d-1 \o 2} \o \Gamma({d-\alpha \o 2})}
  { e^{i(\text{\bf{w}},\text{\bf{z}})} \o (4\pi t)^{d-1 \o 2}}
  e^{-|\text{\bf{w}}|^2/4t} d\text{\bf{w}} {e^{-(u_1-y_1)^2/4t} dt \o t^{1+{1-\alpha \o 2}}} \nonumber \\
  &=& {2^{\alpha-1} \pi^{d-1 \o 2} \o \Gamma({d-\alpha \o 2})}
   \int_0^{\infty} e^{-t|\text{\bf{z}}|^2}
   {e^{-(u_1-y_1)^2/4t} dt \o t^{1+{1-\alpha \o 2}}}=
   {2^{\alpha+1 \o 2} \pi^{d-1 \o 2} \o \Gamma({d-\alpha \o 2})}
    |\text{\bf{z}}|^{1-\alpha \o 2}
   { K_{1-\alpha \o 2}(|u_1-y_1| |\text{\bf{z}}|) \o |u_1-y_1|^{1-\alpha \o 2}} \,.
    \end{eqnarray}
  Thus, the transform of the left-hand side of the formula \pref{wym} takes the form
  \begin{equation*}
   e^{i(\text{\bf{x}},\text{\bf{z}})}
  C  \int_{-\infty}^0  C_{\alpha}^1
       \left({x_1 \o -u_1} \right)^{\alpha/2}
       { e^{-(x_1-u_1)|\text{\bf{z}}|} \o x_1-u_1}
       |\text{\bf{z}}|^{1-\alpha \o 2}
       { K_{1-\alpha \o 2}(|u_1-y_1| |\text{\bf{z}}|) \o |u_1-y_1|^{1-\alpha \o 2}} \,du_1\,.
    \end{equation*}
 The constant $C$ is equal to
 $ 2^{\alpha+1 \o 2} \pi^{d-1 \o 2} / \Gamma({d-\alpha \o 2})$.
 The transform of the right-hand side is of the form
 \begin{equation*}
 e^{i(\text{\bf{x}},\text{\bf{z}})}
   C
            |\text{\bf{z}}|^{1-\alpha \o 2}
 { K_{1-\alpha \o 2}(|x_1-y_1| |\text{\bf{z}}|) \o |x_1-y_1|^{1-\alpha \o 2}} \,.
 \end{equation*}
 After multiplication of both sides by
  $C^{-1} e^{-i(\text{\bf{x}},\text{\bf{z}})}e^{x_d |\text{\bf{z}}|}  $ and
   $|\text{\bf{z}}|^{\alpha -1 \o 2}$ we see that the formula \pref{wym} takes the form
  \begin{equation} \label{wym1}
  {\sin(\pi \alpha /2) \o \pi}  \int_{-\infty}^0
     \left({x_1 \o -u_1} \right)^{\alpha/2}
   {\Phi(u_1) \o x_1 - u_1} du_1 = \Phi(x_1)\,,
   \end{equation}
 where
 \begin{equation}  \label{harmonic}
 \Phi(u) = e^{u|\text{\bf{z}}|}
  { K_{1-\alpha \o 2}(|u+r| |\text{\bf{z}}|) \o |u+r|^{1-\alpha \o 2}}  \,, \quad r=-y_1>0\,.
 \end{equation}

 In another words, to prove \pref{wym} it is enough to show that the function $\Phi$
 is $\alpha$-harmonic on the half-line $(0,\infty)$, for every $r>0$ and this is proved
 in the lemma below.
 \end{proof}

  \begin{lem} \label{Cauchyharm}
    The following function $\Phi(u)$ is $\alpha$-harmonic on $(0, \infty)$   
    (harmonic with respect to isotropic $\alpha$-stable L\'evy process):
    \begin{equation*}
    \Phi(u) = 
  e^{u|\text{\bf{z}}|}
  { K_{1-\alpha \o 2}(|u+r| |\text{\bf{z}}|) \o |u+r|^{1-\alpha \o 2}}  \,, \quad r>0\,.   
    \end{equation*}
   \end{lem}
    \begin{proof}
    We provide the proof in the simplest case, that is, for $\alpha=1$.
    
    To check this, we apply again Theorem \ref{pkf} or, more specifically, the
    transformation rule allowing to derive the $1$-stable one-dimensional Poisson kernel for the set $(0,\infty)$ from the Poisson kernel of half-space for two-dimensional Brownian motion on the (complex) plane. We follow the notation introduced in this theorem.

    Recall that we identified the half-axis  $(-\infty,0]$ with the negative real
    coordinate axis, and that $Z_t$ was the
    two-dimensional complex Brownian motion on the (complex) plane. Also $Y=Z^2$. Then for
    $w=x+iy$, $w \in D=((-\infty,0]\times \{0\})^c$, regarded as the starting point of $Y$ we have
    \begin{equation*}
    E^w \Phi(Y_{\tau_D}) =  E^{w^{1/2}} \Phi(Z_{\tau_{\H}}^2) \,,
     \end{equation*}
     for integrable functions $\Phi$ on the plane.
    In another words, we obtain
    \begin{equation*}
    \int_{-\infty}^{0} h(w,u)\,\Phi(u)\,du =
    \int_{-\infty}^{\infty} \,{|\Re{w^{1/2}}| \o |iv-w^{1/2}|^2}\,\Phi((iv)^2)\, dv\,.
    \end{equation*}
    Observe that the right-hand side of the above equation is the usual Poisson integral
    of the function $\Phi((iv)^2) = \Phi(-v^2)$. In our case, the function
    $\Phi(u)=e^{u |\text{\bf{z}}|} K_0(|r+u||\text{\bf{z}}|)$, and it is enough to show
    that   $\Phi(u^2)$ extends to the classical (Newtonian) harmonic function for all
    $z \in \H$. To do this, note that the (complex-valued) function
     $\Phi(z^2)=e^{z^2 x} K_0((r+z^2)x)$, $x=|\text{\bf{z}}|>0$, is analytic for all $z \in \H$
     so the (real-valued) function $\Re{\Phi(z^2)} +\Im{\Phi(z^2)}$ is harmonic over $\H$.
     The boundary value of the above function (for $z=iv$) coincides with
    $\Phi(-v^2)=e^{-v^2 |\text{\bf{z}}|} K_0(|r-v^2||\text{\bf{z}}|)$. To realize that
    we apply the following integral formula for the functions $K_{\nu}$, specialized
    for $\nu=0$ (see \cite{GR}, p. 907):
    \begin{equation*}
   K_0(zx)= (2z)^{-1/2} e^{-zx} \int_0^{\infty} e^{-tx}\,t^{-1/2}\,
   \(1+{t \o z^2}\)^{-1/2}\,dt\,,\quad |\arg(z)|<\pi\,,\quad x>0\,,
    \end{equation*}
   where all real-valued square roots are taken to be arithmetic (positive).
   It is also apparent that for $z$ from the positive real axis the value of the function
   $\Phi(z^2)$ coincides with the above defined harmonic function. This observation ends
   the proof of the case $\alpha=1$. The general case is proved in \cite{BMR}, using complex variables methods. For the proof of the case $d=1$ see \cite{W}.
    \end{proof}

\subsection{Joint distribution of $(\tau_D,B_{\tau_{D}})$.}
Recall that we consider the two-dimensional
Brownian motion $(B(t),W(t))_{t\geq 0}$ starting from the point $(x,0)$, $x>0$,
of the positive part of the horizontal axis and hitting the negative part of the
same axis. We have, as before
\begin{equation} \label{setD}
D=((-\infty,0]\times \{0\})^{c} \quad \text{and} \quad
T_D=\inf\{t>0; (B(t),W(t))\in D\}.
\end{equation}

Now, denote the joint density of the distribution of
$(\tau_D,B_{\tau_D})$ with respect to $P^{(x,0)}$ 
by $h_x(s,z)$. We determine this density using the form of the $d$-dimensional $1$-stable (Cauchy)
Poisson kernel
of the half-space.

We show first that the $1$-stable $d$-dimensional Poisson kernel for the set
$\H = \{ (x_1,{\bf x}) \in \R^d; x_1>0 \}$ is determined by the joint density
of $(\tau_D,B_{\tau_D})$. To do this, let $B_t= (B_t^{(1)}, \ldots , B_t^{(d)})$
be the $d$-dimensional Brownian motion and let  $\eta_t$ be the standard
 $1/2$-stable subordinator independent from $B$.
  The process $Y_t = B_{\eta_t}$ is our isotropic $1$-stable
 (Cauchy) L\'evy motion.

Furthermore, let $W_t$ be another, $1$-dimensional Brownian motion, independent from $B$ and
let $L_s$ be its local time at $0$. We represent the process $\eta_t$ as the
right-continuous "inverse" of the process $L_s$:
\begin{equation*}
 \eta_t = \inf \{ s>0; L_s \geq t\}\,.
 \end{equation*}
 Then the process $L$ grows exactly on the set of zeros of $W$.
Moreover, if we define
$$\tau_0 = \inf \{t>0; B_{\eta_t}^1<0\}$$
 then
 $B_{\eta_{\tau_0}}^1<0$ and   $B_{\eta_{\tau_0}-}^1>0$ which means that
 $(\eta_{\tau_0 -},\eta_{\tau_0})$ is one of the maximal intervals at which $L$ is constant.
 The process $W_t$ makes a single excursion at this interval and takes the value $0$
 at its endpoints.

 Summarizing,
 we obtain
 \begin{eqnarray*}
 \eta_{\tau_0} &=& \inf \{\eta_t ; B_{\eta_t}^1<0\}  \\
 =\inf \{ s=\eta_t; B_s^1<0\} &=&
 \inf \{ s>0; (B_s^1,W_s) \in ((-\infty,0]\times \{0\})\}= \tau_{D}\,. \\
 \end{eqnarray*}

 We now compute for ${\bf u} =(u_2, \cdots, u_{d})$ and arbitrary bounded Borel
 function $f$:
 \begin{eqnarray}   \label{BTD}
\nonumber E^{({\bf x},y)} [f(B_{\eta_{\tau_0}}^1)\,\exp i(u_2\,B_{\eta_{\tau_0}}^2 +
 \cdots + u_{d}\,B_{\eta_{\tau_0}}^{d})]
 &=& e^{i({\bf u},{\bf x})}
 E^y[f(B_{\eta_{\tau_0}}^1)\,e^{-\eta_{\tau_0}|{\bf u}|^2/2}] \\ 
 &=& e^{i({\bf u},{\bf x})}E^{(y,0)}[f(B_{T_D}^1)\,e^{-\tau_D|{\bf u}|^2/2}]\,.
 \end{eqnarray}

We now determine the joint density of $(\tau_D,B_{\tau_D})$ by the form of
 $d$-dimensional $\alpha$-stable Poisson kernel for $\alpha=1$. 

The starting point is now the formula \pref{dp}. We rewrite it in a form suitable
for further calculations:

The $\alpha$-stable Poisson kernel of the set $\H$ is of the form:
 \begin{equation*}
 P((x_1,\text{\bf{x}}),z)=C_{\alpha}^d \left({x_1 \o -z_1} \right)^{\alpha/2}
  {dz \o (|\text{\bf{z}}-\text{\bf{x}}|^2 +(x_1-z_1)^2)^{d/2}},\quad
 z_1<0<x_1,
 \end{equation*}
 where $C_{\alpha}^d = \Gamma(d/2) \sin(\pi \alpha/2)/ \pi^{1+d/2} $.

  We assume in what
 follows that $d>1$ and that $\alpha=1$.

 We take
 the $(d-1)$-dimensional Fourier transform with respect to $\text{\bf{x}}\in \R^{d-1}$.

 Observe that taking into account the form of the $(d-1)$-dimensional Fourier
 transform of symmetric Cauchy distribution in $\R^{d-1}$ we obtain
 \begin{equation} \label{Cauchy}
 \int_{\R^{d-1}} {\Gamma(d/2) \o \pi^{d/2}}
 {(x_1-z_1) e^{i(\text{\bf{v}},\text{\bf{z}})} \o (|\text{\bf{v}}|^2 + (x_1-z_1)^2)^{d/2}}\
 d\text{\bf{v}} = e^{-(x_1-z_1)|\text{\bf{z}}|} \  .
 \end{equation}

Hence, taking into account \pref{dp} and \pref{Cauchy} we obtain

\begin{eqnarray*}
&{}& E^{(x_1,{\bf x})} [f(B_{\eta_{\tau_0}}^1)\,\exp i(u_2\,B_{\eta_{\tau_0}}^2 +
 \cdots + u_{d}\,B_{\eta_{\tau_0}}^{d})]\\
 &=&  \int_{\R^d} e^{i({\bf u}, {\bf z})}\,f(z_1)\,P((x_1,{\bf x}),z)\,dz\\
 &=& {1 \o \pi} \int_{-\infty}^0 \({x_1 \o -z_1}\)^{1/2}\, {1 \o x_1-z_1}\,f(z_1)
 \int_{\R^{d-1}} {\Gamma(d/2) \o \pi^{d/2}}
  {(x_1-z_1) e^{i(\text{\bf{u}},\text{\bf{z}})} \o (|\text{\bf{z}}-\text{\bf{x}}|^2 + (x_1-z_1)^2)^{d/2}}\
  d\text{\bf{z}}\,dz_1 \\
 &=&  e^{i(\text{\bf{u}},\text{\bf{x}})}
   {1 \o \pi} \int_{-\infty}^0 \({y \o -z_1}\)^{1/2}\, {1 \o x_1-z_1}\,
   e^{-(x_1-z_1)|\text{\bf{u}}|}\,f(z_1)\, dz_1\,.
 \end{eqnarray*}
Comparing \pref{BTD} we obtain

\begin{equation} \label{GBTD}
E^{(x,0)}[e^{-\lambda^2 \tau_D}\,f(B_{\tau_D}^1)]=
{1 \o \pi} \int_{-\infty}^0 \({x \o -z}\)^{1/2}\, {e^{-\lambda (x-z)}\o x-z}\,
   \,f(z)\, dz\,.
\end{equation}

 At the same time, for the $1/2$-stable subordinator $\eta_t$ we know the explicit
 form of the density so we can write

 \begin{equation} \label{eta}
 E^0 e^{-\lambda \eta_t} = {1 \o 2 \sqrt{\pi}}
 \int_0^{\infty} e^{-\lambda s}\, {t \o s^{3/2}}\, e^{-t^2 /4s}\,ds = e^{-t\lambda^{1/2}}\,.
 \end{equation}

 Taking into account \pref{eta} we transform the last integral in the previous
 calculations into the form

 \begin{equation*}
 e^{i(\text{\bf{u}},\text{\bf{x}})}
 \int_{-\infty}^0 \int_0^{\infty} {1 \o \pi}\,\({x_1 \o -z_1}\)^{1/2}\,
 e^{-|\text{\bf{u}}|^2 s} {1 \o 2 \sqrt{\pi}}\, {e^{-(x_1-z_1)^2/4s} \o s^{3/2}}\, f(z_1)\,ds\,dz_1\,.
 \end{equation*}

 The last formula and \pref{BTD} yield the form of the joint density function
 of $(\tau_D,B_{\tau_D})$:

 \begin{equation} \label{BTDdensity}
 h_x(s,w) =  {1 \o \pi}\,\({x \o -w}\)^{1/2}\,
 {1 \o 2 \sqrt{ \pi}}\, {e^{-(x-w)^2/4s} \o s^{3/2}}\,,
 \quad {\text{under}}\quad P^{(x,0)} \,.
 \end{equation}

  Observe that this density is the product of the density of $B_{T_D}$ by the
  density of $\eta_{x-w}$.

We summarize our result in the following proposition. For the sake of further calculations we
also record well-known facts concerning  hitting distributions of our Brownian motion.

By $\tau$ we denote the first hitting time of the point $0$
by the process $W(t)$; by  $\sigma$ - the first hitting time of the point $0$
by the process $B(t)$. For the sake of convenience we also record joint densities
of $(\tau,B_{\tau})$ and $(\sigma,W_{\sigma})$, denoted by $g_x^{\tau}(s,w)$
 or $g_x^{\sigma}(s,w)$, respectively. Here $x=(x_1,x_2)$ is the starting point
 of $(B(t),W(t),)$. The (marginal) density of the variables $\tau$,  $\sigma$ is
 denoted by $g_x(s)$.  By $p(t,x,w)$ we denote the transition density
 (depending on the context) of (one or two - dimensional) Brownian motion.
 The crucial (and well-known) observation stemming out of the reflection principle
 for one dimensional Brownian motion starting from the point $y>0$
 is the following: $g_y(s)=(y/s)p(s,y,0)$. This at once yields the formulas for
 the joint densities of $(\tau,B_{\tau})$ and $(\sigma,W_{\sigma})$.
 \begin{prop} \label{btautau}
 We have for $z<0<x$ and $s>0$
 \begin{equation} \label{bttd}
 h_x(s,z)= {1 \o \pi} \({x \o -z}\)^{1/2} {s^{-3/2} \o 2 \sqrt{ \pi}}
 e^{-(x-z)^2/4s} = {1 \o \pi} \({x \o -z}\)^{1/2} {1 \o s}\, p(s,x,z) \,,
 \quad {\text{under}}\quad P^{(x,0)} \,,
 \end{equation}
 \begin{equation} \label{btt}
 g_{(x,y)}^{\tau}(s,w) = {|y| \o s}\, p(s,(x,y),(w,0))\,,\quad {\text{under}}\quad P^{(x,y)} \,,
  \end{equation}
 \begin{equation} \label{bss}
  g_{(x,y)}^{\sigma}(s,w) = {|x| \o s}\, p(s,(x,y),(0,w))\,,\quad {\text{under}}\quad P^{(x,y)} 
\,.
   \end{equation}
  \end{prop}
 An intriguing observation about the density $h_x(z,s)$ is contained in the following

  \begin{cor} \label{btautau1}
  We have for $z<0<x$ and $s>0$
  \begin{equation} \label{bttd1}
  h_x(s,z)= h(x,z) g_{(x-z)}(s) \,.
  \end{equation}
  \end{cor}
 We recall that $ g_{a}(s)$ denotes the density of $\tau$ with respect to $P^a$.
 
 When the process starts outside the positive axis situation is much more complicated. Namely, we have the following, see \cite{BMR}:
 
 \begin{thm} \label{taudbtaud}
 \begin{equation*}
 E^{\bf{z}}[e^{-\lambda^2 \tau_D/2};\,B(\tau_D) \in dw]
 = \frac{1}{\sqrt{2}\pi}\frac{(|{\bf{z}}|+z_1)^{1/2}}{\sqrt{|w|}}
 \int_0^{\infty}t\, 
 \frac{e^{-(|{\bf{z}}|+|w|)\sqrt{t^2 +\lambda^2}}}{\sqrt{t^2 +\lambda^2}}\,
 \cosh(\sqrt{2|w|}\sqrt{|{\bf{z}}|-z_1}\,t)\,dt\,.
 \end{equation*}
 \end{thm} 
 From this theorem we can recover the joint distribution of $(\tau_D,B(\tau_D))$:
 
 \begin{cor}
 \begin{equation*}
 g_{\bf{z}}(s,w) =  \frac{|({\bf{z}}|+z_1)^{1/2}}{2 \pi^{3/2}(|w|s)^{1/2}}\,
 e^{-(|{\bf{z}}|+|w|)^2/2s}
 \int_0^{\infty}t\,e^{-t^2 s/2}\,\cosh(\sqrt{2|w|}\sqrt{|{\bf{z}}|-z_1}\,t)\,dt\,.
 \end{equation*}
 \end{cor} 
 \begin{proof}
 From the preceding theorem and the representation of $K_{1/2}(x)=\sqrt{\pi/2x}\,e^{-x}$ we obtain
 \begin{equation*}
 \frac{e^{-\sqrt{t^2+\lambda^2}\,(|{\bf{z}}|+|w|)}}{\sqrt{t^2+\lambda^2}}=
 \frac{1}{\sqrt{2\pi}}\int_0^{\infty} e^{-(t^2+\lambda^2)s/2}\,
 e^{-(|{\bf{z}}|+|w|)^2/2s}\,s^{-1/2}\,ds\,.
 \end{equation*}
 This and the formula from the theorem yield the result.
 \end{proof} 
The Theorem \ref{taudbtaud} enables us to express the value of the 
{\it conditional gauge} $u({\bf{z}},w)$, 
defined as $E_w^{\bf{z}}[e^{-\lambda^2 \tau_D/2}]$
for ${\bf{z}}\in D$ and $w<0$.
 We recall that $E_x^{\bf{z}}$ denotes the conditional expectation with respect to the density function of 
  the distribution $(B(\tau_D),W(\tau_D)) = (B(\tau_D),0)$ which we denoted by $h({\bf{z}},w)$. For the Brownian motion
  it is equivalent to the usual conditioning by the random variable $B(\tau_D)$. 
  We thus have
\begin{thm}
For ${\bf{z}} \in D$, $w<0$ we obtain the following formula
\begin{eqnarray*} 
 &{}&E_w^{\bf{z}}[e^{-\lambda^2 \tau_D/2}]
 = \frac{(z_1-w)^2+z_2^2}{|{\bf{z}}|+|w|} \int_0^{\infty}t\, 
 \frac{e^{-(|{\bf{z}}|+|w|)\sqrt{t^2 +\lambda^2}}}{\sqrt{t^2 +\lambda^2}}\,
 \cosh(\sqrt{2|w|}\sqrt{|{\bf{z}}|-z_1}\,t)\,dt\\
 &=&\frac{(z_1-w)^2+z_2^2}{(|{\bf{z}}|+|w|)^2} \left[e^{-(|{\bf{z}}|+|w|)\lambda}
+\sqrt{2|w|(|{\bf{z}}|-z_1)}\,
\int_0^{\infty} 
 e^{-(|{\bf{z}}|+|w|)\sqrt{t^2 +\lambda^2}}\,
 \sinh(\sqrt{2|w|}\sqrt{|{\bf{z}}|-z_1}\,t)\,dt\right].\\
\end{eqnarray*}
\end{thm}
\begin{proof}
Computing the function $u({\bf{z}},w)$ as conditional expectation we obtain
\begin{eqnarray*}
&{}& u({\bf{z}},w)=
 E_w^{\bf{z}}[e^{-\lambda^2 \tau_D/2}]=
E^{\bf{z}}[e^{-\lambda^2 \tau_D/2}|\,B(\tau_D) = w]\\
&=&
\frac{E^{\bf{z}}[e^{-\lambda^2 \tau_D/2};\,B(\tau_D) \in dw]}
{E^{\bf{z}}[B(\tau_D) \in dw]}
= \frac{ \int_0^{\infty}t\, 
 \frac{e^{-(|{\bf{z}}|+|w|)\sqrt{t^2 +\lambda^2}}}{\sqrt{t^2 +\lambda^2}}\,
 \cosh(\sqrt{2|w|}\sqrt{|{\bf{z}}|-z_1}\,t)\,dt}
{\int_0^{\infty} 
 e^{-(|{\bf{z}}|+|w|)t}\,
 \cosh(\sqrt{2|w|}\sqrt{|{\bf{z}}|-z_1}\,t)\,dt}\\
&=& \frac{(z_1-w)^2+z_2^2}{|{\bf{z}}|+|w|} \int_0^{\infty}t\, 
 \frac{e^{-(|{\bf{z}}|+|w|)\sqrt{t^2 +\lambda^2}}}{\sqrt{t^2 +\lambda^2}}\,
 \cosh(\sqrt{2|w|}\sqrt{|{\bf{z}}|-z_1}\,t)\,dt\\
&=& \frac{(z_1-w)^2+z_2^2}{(|{\bf{z}}|+|w|)^2} \left[e^{-(|{\bf{z}}|+|w|)\lambda}
+\sqrt{2|w|(|{\bf{z}}|-z_1)}\,
\int_0^{\infty} 
 e^{-(|{\bf{z}}|+|w|)\sqrt{t^2 +\lambda^2}}\,
 \sinh(\sqrt{2|w|}\sqrt{|{\bf{z}}|-z_1}\,t)\,dt\right].\\
\end{eqnarray*}
\end{proof}

We recall that the distribution of $B(\tau_D)$, denoted by $h({\bf{z}},w)$ is given by the formula:
\begin{equation*}
h({\bf{z}},w)=
\frac{1}{\sqrt{2}\pi}\frac{(|{\bf{z}}|+z_1)^{1/2}}{\sqrt{|w|}}
\frac{|{\bf{z}}|+|w|}{(z_1-w)^{1/2}+z_2^2}\,.
\end{equation*}

 
 
  In what follows we apply $(d-1)$-dimensional Fourier transform for various objects
  pertaining our processes. We begin with the Gaussian case

  \begin{lem}
  \begin{eqnarray*}
  g_u(x,\widehat{(\cdot},y_d))(\bf{z}) &=& e^{i({\bf{z}},{\bf{x}})}
  \frac{e^{-|{\bf{z}}|^2\,u}}{\sqrt{4\pi\,u}}
  \exp{\left(-\frac{(x_d-y_d)^2}{4u}\right)}\\
  &=&  e^{i(\bf{z},{\bf{x}})} e^{-|{\bf{z}}|^2\,u}
   g_u(x_d,y_d)
  \,.
  \end{eqnarray*}
  \end{lem}

  Applying the above relation for $d$-dimensional $\alpha$-stable process we obtain

    \begin{lem}
     \begin{equation*}
     p_u(x,\widehat{(\cdot},y_d))({\bf{z}}) = e^{i({\bf{z}},{\bf{x}})}
     e^{-|{\bf{z}}|^{\alpha}\,u}\,
      p_u^{|{\bf{z}}|^{\alpha}} (x_d,y_d)\,.
     \end{equation*}
     \end{lem}
 We recall that $p_u^{m} (x,y)$ denotes the transition
  density function of the $1$-dimensional $\alpha$-stable relativistic process
 with the parameter $m$.
 Thus, the application of $(d-1)$-dimensional Fourier transform on the
 $d$-dimensional $\alpha$-stable process produces  the $1$-dimensional
 $\alpha$-stable relativistic process with the parameter $m=|{\bf{z}}|^{\alpha}$
 where ${\bf{z}}$ is the argument of the Fourier transform.

 Taking into account the relation \pref{BTD} and the above calculations we obtain

  \begin{eqnarray*}
   E^{({\bf x},y)} [f(B_{\eta_{\tau_0}}^1)\,\exp i(u_2\,B_{\eta_{\tau_0}}^2 +
   \cdots + u_{d}\,B_{\eta_{\tau_0}}^{d})]
   &=&  e^{i({\bf u},{\bf x})}E^{(y,0)}[f(B_{\tau_D}^1)\,e^{-\tau_D|{\bf u}|^2/2}] \\
   &=& e^{i({\bf u},{\bf x})}E^{y}[f(X_{\rho}^{|{\bf u}|})\,
   e^{-\rho|{\bf u}|}]\,.
   \end{eqnarray*}
  Here $\rho$ is the first exit time of the relativistic Cauchy process
  $X_t^{|{\bf u}|}$   with the parameter ${|{\bf u}|}$ from the
   halfline
  $(0,\infty)$. 
   Comparing with the previous calculations we obtain 
   \begin{thm}
   We obtain
   \begin{eqnarray*}
   E^x[f(X_{\rho}^{|{\bf u}|})\,e^{-\rho|{\bf u}|}]
   &=&
   E^{(x,0)}[f(B_{\tau_D}^1)\,e^{-\tau_D|{\bf u}|^2/2}]\\
   =\int_{-\infty}^{0} f(z)\{\int_0^{\infty} e^{-|{\bf u}|^2 s/2} h_x(s,z)\,ds\} dz
   &=& \int_{-\infty}^{0} f(z) \frac{1}{\pi} \left(\frac{x}{-z}\right)^{1/2}
        \,\frac{e^{-|{\bf u}|(x-z)}}{x-z}\,dz.
   \end{eqnarray*}
   \end{thm}
  As a corollary we exhibit an explicit form of the density function of relativistic Cauchy $1$-Poisson kernel:

  \begin{cor}
  The density function of   $E^{x}[e^{-\rho}\,f(X^{1}(\rho))]$ is of the form
  \begin{equation*}
  \int_0^{\infty} e^{-s}\,h_x(s,z)\,ds
 =h(x,z)\,e^{-(x-z)}
 =h(x,z)\,E_{(z,0)}^{(x,0)}[e^{-\tau_D}]\,.  
  \end{equation*}
  \end{cor}
  \begin{proof}

  \begin{equation*}
   E_{(z,0)}^{(x,0)}[e^{-\tau_D}]=
  E^{(x,0)}[e^{-\tau_D}|B(\tau_D)=z]=\frac{\int_0^{\infty} e^{-s}\,h_x(s,z)\,ds}{h(x,z)}  
  \end{equation*}
  \end{proof}
  In the case of halfline we obtain
  $$
  \int_0^{\infty} e^{-s}\,h_x(s,z)\,ds=
   \frac{1}{\pi} \left(\frac{x}{-z}\right)^{1/2}
  \,\frac{e^{-(x-z)}}{x-z}\,.
  $$

  Here $ h(x,z)$ denotes the $P^{(x,0)}$density function of $B(\tau_D)$ and is easily identifiable.
   We recall that
   $$h(x,z) = \frac{1}{\pi} \left(\frac{x}{-z}\right)^{1/2}\,\frac{1}{x-z}\,,$$
   where $D=((-\infty,0]\times \{0\})^{c}$ and
   $$h(x,z) = \frac{1}{2\pi} \left(\frac{1-x^2}{z^2-1}\right)^{1/2}\,\frac{1}{|x-z|}\,,$$
   for the case when  
   $D=((-\infty,-1]\times \{0\}\cup [1,\infty) \times \{0\} )^{c}$\,.

 Observe that the value of the conditional gauge $u(x,w)=E_{w}^{(x,0)}[e^{-\tau_D}]$ is surprisingly simple when 
 $D=((-\infty,0]\times \{0\})^{c}$, it is namely equal to $e^{-(x-w)}$. The complete value of the function 
 $u((x,y),w)=E_{w}^{(x,y)}[e^{-\tau_D}]$ is much more complicated. To determine  its value we need once again some information about relativistic process.
 
 
 Denote $\widetilde{h(x,w)\,u(x,w)}=h(x,w)\,u(x,w)+\delta_0(x-w)$.
 \begin{thm}
 We have for $(x,y) \in D$ and $w<0$
 \begin{equation} \label{fgauge}
 h((x,y),w)\,u((x,y),w)=\frac{1}{\sqrt{2}}\,e^{-|y|}\tilde{p}_{|y|}\ast \widetilde{h(\cdot,w)\,u(\cdot,w)}(x) \,,
 \end{equation}
 where $\tilde{p}_{|y|}$ is the Cauchy relativistic stable process with parameter $m=1$ and the 
 Fourier transform of the right-hand side of \pref{fgauge} is of the form
 \begin{equation*} 
 \frac{1}{\sqrt{2}}\, e^{-\sqrt{\xi^2+1}|y|}\, {\cal F}(\widetilde{h(\cdot,w)\,u(\cdot,w)})(\xi)\,.
 \end{equation*}
 
 \end{thm}
 \begin{proof}
 Denote
 \begin{equation*}
 \Psi((x-z),y)=
 \int_0^{\infty}e^{-s}\,g_{(x,y)}^{\tau}(s,z) \,ds
 \end{equation*}
 
 We begin with computing  Fourier transform of $\Psi((x-z),y)$
 \begin{eqnarray*}
 \int_{\R} e^{i\xi x}\int_0^{\infty}e^{-s}\,g_{(x,y)}^{\tau}(s,z) \,ds\,dx &=&\int_0^{\infty}e^{-t}\,\frac{|y|}{t}
 \int_{\R}e^{i\xi x} p(t,(x,y),(z,0))\,dx\,dt\\
 =e^{i\xi z}\frac{|y|}{2\sqrt{\pi}} \int_0^{\infty} e^{-t}\,t^{-3/2}\,e^{-\frac{y^2}{4t}}\,e^{-t\xi^2}\,dt
 &=& e^{i\xi z}\frac{|y|^{1/2}}{2\sqrt{\pi}} 
 (1+\xi^2)^{1/4}\,K_{-1/2}\left((1+\xi^2)^{1/2}|y|\right)\\
 =\frac{e^{i\xi z}}{\sqrt{2}}\, e^{-\sqrt{\xi^2+1}|y|}&=& \frac{e^{i\xi z}}{\sqrt{2}}\,
 {\cal F}(e^{-|y|}\tilde{p}_{|y|})\,.\\
 \end{eqnarray*}
 On the other hand
 \begin{equation*}
 \int_0^{\infty} e^{-s}\,h_{(z,0)}(s,w)\,ds = h(z,w)\,u(z,w)\,.
 \end{equation*}
 Thus, we obtain for $D= ((-\infty,0] \times \{0\})^{c}$
 \begin{eqnarray*}
 &{}&E^{(x,y)}[f(B_{\tau_D})\,E_{B_{\tau_D}}^{(x,y)}[e^{-\tau_D}]] =
 E^{(x,y)}[f(B_{\tau_D})\,e^{-\tau_D}]\\
 &=&E^{(x,y)}[\tau<\tau_D;\,e^{-\tau}\,E^{(B_{\tau},0)}[e^{-\tau_D}\,f(B_{\tau_D})]] +
 E^{(x,y)}[\tau=\tau_D;\,e^{-\tau_D}\,f(B_{\tau_D})]\\
 &=& \int_{0}^{\infty}\int_{0}^{\infty}e^{-t}\,g_{(x,y)}^{\tau}(t,z)\,
 E^{(z,0)}[f(B_{\tau_D})\,e^{-\tau_D}]\,dt\,dz
 + \int_{0}^{\infty}\int_{0}^{\infty}e^{-s}\,g_{(x,y)}^{\tau}(s,w)\,f(w)\,ds\,dw\\
 &=&\int_{0}^{\infty}\int_{0}^{\infty}e^{-t}\,g_{(x,y)}^{\tau}(t,z)\,
 \int_{0}^{\infty}\int_{-\infty}^{0} e^{-s}\,h_{(z,0)}(s,w)\,f(w)\,ds\,dw\,dt\,dz\,\\
  &+&
 \int_{0}^{\infty}\int_{0}^{\infty}e^{-s}\,g_{(x,y)}^{\tau}(s,w)\,f(w)\,ds\,dw \\
 &=&\int_{-\infty}^0 \int_{0}^{\infty} \left( \int_{0}^{\infty}e^{-t}\,g_{(x,y)}^{\tau}(t,z)\,dt\right)
 \left( \int_0^{\infty} e^{-s}\,h_{(z,0)}(s,w)\,ds\,\right)\,f(w)\,dw\\
 &+&
 \int_{-\infty}^0\left(\int_{0}^{\infty}e^{-s}\,g_{(x,y)}^{\tau}(s,w)\,ds\right)\,f(w)\,dw \\
 &=&\int_{-\infty}^0 \left(\int_{0}^{\infty} \Psi((x-z),y)\,h(z,w)\,u(z,w)\,dz
  + \Psi((x-w),y)\right)\,f(w)\,dw\,. \\
 \end{eqnarray*}
 The same calculations apply for $(-\infty,-1] \times \{0\} \cup [1,\infty)\times \{0\}$ with the exception that in the last integral we integrate over $|w|\geq 1$ instead of $w \leq 0$.
 
 \end{proof}
 In the same way we can express the Green function of $D$.
 
 \begin{thm}
 We obtain
 \begin{equation*}
 G_{D}((x,y),(z_1,z_2))=G_{\H}((x,y),(z_1,z_2))+ \frac{1}{\sqrt{2}}\,e^{-|y|}\,\tilde{p}_{|y|}
 \ast G_{D}(\cdot,(z_1,z_2))(x)\,.
 \end{equation*}
 When $(z_1,z_2)=(z_1,0)$ with $z_1>0$ then we obtain
 \begin{equation*}
 G_{D}((x,y),(z_1,0))= \frac{1}{\sqrt{2}}\,e^{-|y|}\,\tilde{p}_{|y|}
 \ast G_{D}(\cdot,z_1)(x)\,.
 \end{equation*}

 \end{thm}

 The formula for Poisson kernel of the set $D$ enables us to write down
 a more  explicit form of the $\lambda$-Green function of this set.
  For convenience we take $\lambda^2/2$ instead of $\lambda$. As a corollary we write the transition density of our process but only in the case when both arguments are in the horizontal line. 
 
 We recall that $\text{\bf{x}}=(x_1,x_2)$ and the same convention applies to $\text{\bf{y}}$. We also denote $\text{\bf{y}}^{\ast}=(y_1,-y_2)$
for $\text{\bf{y}}=(y_1,y_2)$.
 \begin{thm}[$\lambda^2/2$-Green function] We have the following formula
 \begin{eqnarray} \label{lGreen0}
 G_{D}^{\lambda^2/2}(\text{\bf{x}},\text{\bf{y}}) &=& \frac{\lambda^2\,|x_2\,y_2|}{\pi^2}\int_0^{\infty}\int_0^{\infty} \frac{K_1(\lambda|\text{\bf{x}}-w|)}{|\text{\bf{x}}-w|}
 \frac{K_1(\lambda|\text{\bf{y}}-z|)}{|\text{\bf{y}}-z|}\,
G_{D}^{\lambda^2/2}(w,z)\,dw\,dz\\ 
&+& G_{\H}^{\lambda^2/2}(\text{\bf{x}},\text{\bf{y}})\,.\nonumber
 \end{eqnarray}
We assume here that $|x_2|>0$ and $|y_2|>0$. We also denote $G_{D}^{\lambda^2/2}(x,y):=G_{D}^{\lambda^2/2}((x,0),(y,0))$.
 \begin{equation} \label{lGreen1}
 G_{D}^{\lambda^2/2}((x_1,x_2),(y,0)) = \frac{\lambda\,|x_2|}{\pi^2}\int_0^{\infty} \frac{K_1(\lambda|\text{\bf{x}}-w|)}{|\text{\bf{x}}-w|}\left[\int_{|w-y|}^{w+y}
 \,\frac{e^{-\lambda u}\,du}{(u^2 - |w-y|^2)^{1/2}}\right]\,dw\,,
 \end{equation}
 where $|x_2|, y >0$. For $x_2=0$  we obtain for  $G_{D}^{\lambda^2/2}(x,y)$
 \begin{equation} \label{lGreen2}
 G_{D}^{\lambda^2/2}(x,y) = \frac{1}{\pi} \int_{|x-y|}^{x+y}
 \,\frac{e^{-\lambda u}\,du}{(u^2 - |x-y|^2)^{1/2}}\,,
 \end{equation}
 where $x, y >0$. 
 \end{thm}
 \begin{proof}
Let $f$ be a bounded Borel function on ${\R}^2$. Under the assumption that $|x_2|>0$ we compute the $\lambda^2/2$-Green
operator on $f$
\begin{eqnarray*}
G_{D}^{\lambda^{2}/2} f(\text{\bf{x}}) = 
E^{\text{\bf{x}}}[\int_0^{\tau_D}e^{-\lambda^2 t/2}\,f(B(t),W(t))\,dt]& = &\\
E^{\text{\bf{x}}}[\tau <\tau_D; \int_0^{\tau_D}
e^{-\lambda^2 t/2}\,f(B(t),W(t))\,dt] +
E^{\text{\bf{x}}}[\tau =\tau_D;\int_0^{\tau_D}
e^{-\lambda^2 t/2}\,f(B(t),W(t))\,dt] &=& \\
E^{\text{\bf{x}}}[\tau <\tau_D;E^{(B(\tau),0)}[ \int_0^{\tau_D}
e^{-\lambda^2 t/2}\,f(B(t),W(t))]\,dt] +
E^{\text{\bf{x}}}[\int_0^{\tau}e^{-\lambda^2 t/2}\,f(B(t),W(t))\,dt]&=&\\
\int_0^{\infty} \int_0^{\infty} g_{\text{\bf{x}}}^{\tau}(s,w)\,e^{-\lambda^2 s/2}\,
G_{D}^{\lambda^{2}/2} f((w,0))\,dw\,ds +
G_{\H}^{\lambda^{2}/2} f(\text{\bf{x}})&=& \\
\frac{1}{\pi}\int_0^{\infty} \frac{\lambda\,|x_2|}{|\text{\bf{x}}-w|}\,K_1(\lambda\,|\text{\bf{x}}-w|)\,
G_{D}^{\lambda^{2}/2} f((w,0))\,dw +
G_{\H}^{\lambda^{2}/2} f(\text{\bf{x}})\,.
\end{eqnarray*}
Let us note that the $\lambda^2/2$-harmonic operator of $\H$ is defined as
\begin{eqnarray*}
P_{\partial\H}^{\lambda^2/2}f({\text{\bf{x}}}) &=& \int_{\R} \int_0^{\infty}
e^{-\lambda^2 s/2}\,g_{{\text{\bf{x}}}}^{\tau}(s,w)\,f(w)\,ds\,dw \\
& =& \frac{\lambda|x_2|}{\pi}\int_{\R} 
\frac{K_1(\lambda |{\text{\bf{x}}}-(w,0)|)}{|{\text{\bf{x}}}-(w,0)|}\,f({\text{\bf{x}}})\,dw\,.
\end{eqnarray*}
The equation determining $G_{D}^{\lambda^{2}/2}$ can be thus written as follows
\begin{equation*}
G_{D}^{\lambda^{2}/2}({\text{\bf{x}}},{\text{\bf{y}}})
-G_{\H}^{\lambda^{2}/2}({\text{\bf{x}}},{\text{\bf{y}}})
=P_{\partial\H}^{\lambda^2/2}\,G_{D}^{\lambda^{2}/2}(\cdot,{\text{\bf{y}}})
({\text{\bf{x}}})\,.
\end{equation*}
Applying this equation once again for 
$G_{D}^{\lambda^{2}/2}((w,0),{\text{\bf{y}}})$ as for the function of
${\text{\bf{y}}}$ and taking into account that $G_{\H}^{\lambda^{2}/2}((w,0),{\text{\bf{y}}})=0$ we obtain the conclusion.

Observe that if $x_2=0$ then we already have $\tau<\tau_D$ so the corresponding terms vanish. By Fubini's theorem and from 
Theorem 3.3 in the paper \cite{BMR} with $\alpha=1=d$ and $m=\lambda$ we obtain the formula \pref{lGreen1}. This formula and the calculations above yield \pref{lGreen0}. The formula \pref{lGreen2} is read from the paper \cite{BMR}. 

 \end{proof}
 \begin{cor}[Transition density of killed process] We obtain
  \begin{equation} \label{tkilled1}
 p_{D}(t;x,y) = \frac{e^{-(x-y)^2/2t}}{\sqrt{2}(\pi t)^{3/2}}
 \int_{0}^{2\sqrt{xy}} e^{-u^2/2t}\,du\,.
 \end{equation}
 For $x_2\neq 0$ we obtain
 \begin{equation} \label{tkilled1}
 p_{D}(t;{\text{\bf{x}}},y) = \int_0^{\infty}p_{D}(\cdot;w,y)\,\ast
 g_{\text{\bf{x}}}^{\tau}(\cdot,w)\,dw\,.
 \end{equation}
 The convolution above is uderstood as with respect to semigroup parameter.
 \end{cor}
 \begin{proof}
 Taking into account the definition of the modified Bessel function of the second
  kind $K_{-1/2}$ we obtain
  \begin{equation*}
  e^{-\lambda u} = \sqrt{\frac{2}{\pi}} \,\left(\lambda u\right)^{1/2}
  K_{-1/2}(\lambda u)=
  \frac{u}{\sqrt{2\pi}}\, \int_0^{\infty} 
  e^{-\lambda^2 t/2}\,e^{-u^2/2t}\,t^{-3/2}\,dt\,.
  \end{equation*}
  so putting this into formula \pref{lGreen2} we obtain, after changing the order of integration:
 \begin{equation*}
G_{D}^{\lambda^2/2}(x,y) = \frac{1}{\sqrt{2\pi}\pi}
 \int_0^{\infty} \frac{e^{-\lambda^2 t/2}}{t^{3/2}}
   \int_{|x-y|}^{w+y} 
\frac{e^{-u^2/2t}\,u\,du}{(u^2 - (x-y)^2)^{1/2}}\,dt\,.
 \end{equation*}
 Changing variable $(u^2 -(x-y)^2)^{1/2} = v$ and taking into account that we delt with the $\lambda^2/2$-Green function, being the Laplace transform of
 the transition density of the killed process, the conclusion follows.
 \end{proof}

\begin{cor}\label{Greenfunction}
 
 \begin{equation} \label{GfD1}
 G_{D}((0,y),(x_1,x_2))= {1 \o 2 \pi} \int_{-\infty}^0
 \ln {x_1^2 + (x_2-y)^2 \o x_2^2 + (x_1-z)^2} h(y,z)dz\,.
 \end{equation}
 We also obtain
 \begin{equation} \label{GfD2}
 G_{D}((y,0),(x_1,x_2))= {2 \o 2 \pi^2} \int_0^\infty
 \frac{|x_2|}{x_2^2 + (x_1 - z)^2}
 \ln {\sqrt{z} + \sqrt{y} \o \sqrt{|z-y|}} \,dz\,.
 \end{equation}
 \end{cor}
\begin{proof}
 We begin with justifying the first formula. Observe that the two-dimensional Brownian
 motion is recurrent so we have to work with the compensated potential $U(x,y)=U(x-y)$,
 \begin{equation*}
 U(x)={1 \o \pi} \ln {1 \o |x|} = {1 \o 2 \pi} \ln {1 \o {x_1}^2+{x_2}^2}\,.
 \end{equation*}
 We thus obtain
 \begin{equation*}
 E^{(y,0)}U((B_{\tau_D},0),(x_1,x_2))=\int_{-\infty}^{0} U((z,0),(x_1,x_2))h(y,z)\,dz\,.
  \end{equation*}
 The sweeping-out principle now yields
  \begin{equation*}
 G_D((0,y),(x_1,x_2))=\frac{1}{2\pi}
 \int_{-\infty}^{0} \ln {x_1^2 + (x_2-y)^2 \o x_2^2 + (x_1-z)^2}  h(y,z)\,dz\,.
 \end{equation*}
 To prove the second formula we first observe that the potential $U(\cdot,w)$ is
 a harmonic function on a half-plane, for $w \in (-\infty, \infty) \times \{0\}$
 so we obtain for $(x_1,x_2) \in \H$, with $\H=\{(u,v);v>0 \}$:
 \begin{equation*}
 U((x_1,x_2),(y,0))= E^{(x_1,x_2)} U((B_{\tau},0),(y,0)) =
 {1 \o \pi} \int_{-\infty}^{\infty} {|x_2| \o (x_1-z)^2 + x_2^2} U(z,y) dz\,.
 \end{equation*}
 This and the  sweeping-out formula
 gives the following form  of the Green function
  \begin{eqnarray*}
&{}& G_D((x_1,x_2),(y,0))=E^{(x_1,x_2)}[ U((B_{\tau},W_{\tau}),(y,0)) -
   U((B_{\tau_D},W_{\tau_D}),(y,0))]\\
 &=&E^{(x_1,x_2)}[\tau<\tau_D; U((B_{\tau},0),(y,0)) -
   U((B_{\tau_D},0),(y,0))]\\
 &=& E^{(x_1,x_2)}[B_{\tau}>0;E^{(B_{\tau},0)}[ U((z,0),(y,0)) -
        U((B_{\tau_D},0),(y,0))]|_{z=B_{\tau}}]\\
 &=& E^{(x_1,x_2)}[B_{\tau}>0;G_{(0,\infty)}^Y(B_{\tau},y)]=
 E^{(x_1,x_2)}G_{(0,\infty)}^Y(B_{\tau},y)\,.
\end{eqnarray*}
 Here $G_{(0,\infty)}^Y$ is the Green function of the one-dimensional Cauchy process
$Y_t=B_{\eta_{t}}$ for the half-axis $(0,\infty)$.
 The last equality follows from the fact
that the Green function $G_{(0,\infty)}^Y(v,y)=0$ if $B_{\tau}=v\leq 0$.
Observe that if $x_2=0$ then we obtain
$ E^{(x_1,x_2)}[G_{(0,\infty)}^Y(B_{\tau},y)]=
G_{(0,\infty)}^Y(x_1,y)$ and the formula \pref{Greenfact} holds as well.

To finish the proof, observe that
 the Green function of one-dimensional Cauchy process $Y_t=B_{\eta_{t}}$,
 is of the form (see, e.g. \cite{R})
 \begin{equation*}
 {2 \o \pi} \ln {\sqrt{v} + \sqrt{y} \o \sqrt{|v-y|}} \,.
 \end{equation*}
 This last observation completes the proof.
\end{proof}
{\bf Remark.}
The formula \pref{GfD2} can be written in a more concise form. Namely,
we have the following identity
 \begin{equation} \label{Greenfact}
  G_D((x_1,x_2),(0,y))=E^{(x_1,x_2)}G_{(0,\infty)}^Y(B_{\tau},y)\,.
 \end{equation}

Intuitive meaning of the above formula for $\lambda=0$ is that the value at $(x_1,x_2)$ of
the density function of time spent in the set $D$ by the
process $(W_t,B_t)$, starting from the point $(0,y)$, with $y>0$, is the same as
the density of time spent by the same process starting from the point $(x_1,x_2)$
up to the moment of first hitting time of positive part of horizontal axis and then
behaving as one-dimentional Cauchy process $Y_t$, starting from
the hitting point $B_{\tau}>0$.

The identity \pref{Greenfact} can be used to show that the one-dimensional Green
function determined as
 \begin{equation*}
 \int_{-\infty}^{\infty} G_D((w_1,w_2),(0,y))\,dw_1
 \end{equation*}
 is not well-defined. Indeed, for a positive Borel function $f$ we obtain
 \begin{eqnarray*}
 G_Df(y)&=& E^{(0,y)}\int_0^{\tau_D} f(B_t)\,dt=
 \int_{\R^2}G_D((0,y),(w_1,w_2))f(w_2)\,dw_1 dw_2 \\
 &=&\int_{\R^2}E^{(w_1,w_2)}[G_{(0,\infty)}^Y(B_{\tau},y)]f(w_2)\,dw_1 dw_2 \\
 &=& \frac{1}{\pi}\int_0^{\infty}G_{(0,\infty)}^Y(v,y)\int_{\R^2}{|w_2| f(w_2) \o (w_1-v)^2+w_2^2}\,dv dw_1 dw_2 \\
 &=& \frac{2}{\pi}\int_0^{\infty}G_{(0,\infty)}^Y(v,y)  \int_0^{\infty} P_tf(v) \,dt dv = \infty\,;
  \end{eqnarray*}
 the last equality follows from the fact that $ P_tf(v)=\infty$, since the
 one-dimensional Cauchy motion is reccurent.

 Recall that $\lambda$-Green function (of the set $D$) is defined  as the density
 function of the $\lambda$-Green operator determined by the formula
 \begin{equation} \label{Gf1}
 G_D^{\lambda}f(x)=E^{x}[\int_0^{\tau_D}e^{-\lambda t} f(B_t,W_t)dt]\,.
 \end{equation}
 Conditional $\lambda$-Green operator, $z\in \partial D$ is defined analogously:
\begin{equation} \label{Gf10}
 G_D^{z,\lambda}f(x)=E_z^{x}[\int_0^{\tau_D}e^{-\lambda t} f(B_t,W_t)dt]\,.
 \end{equation}

 Since $p_z^{D}(t,x,w)= h(x,z)^{-1}p^{D}(t,x,w)h(w,z)$ we have the following relationship:

 $$G_D^{z,\lambda}(x,w)=h(x,z)^{-1}G_D^{\lambda}(x,w)h(w,z)$$

 \begin{lem} \label{gauge01}
 \begin{equation} \label{g01}
 \lambda \int_D  G_D^{z,\lambda}(x,w) dw =
  1-E_z^{x} e^{-\lambda \tau_D}\,.
  \end{equation}
  \end{lem}

  \begin{proof}

 \begin{eqnarray*}
  \lambda \int_D  G_D^{z,\lambda}(x,w) dw&=&\lambda \int_D\int_0^\infty e^{-\lambda t}p_z^{D}(t,x,w)dtdw\\
  &=&\lambda \int_0^\infty \int_D e^{-\lambda t}p_z^{D}(t,x,w)dwdt\\
  &=&\lambda \int_0^\infty e^{-\lambda t}  P_z^{x}(\tau_D\ge t)dt=1-E_z^{x} e^{-\lambda \tau_D}
\end{eqnarray*}
 \end{proof}

 Let $e_q(t)=e^{\int_0^tq(X_s)ds}$ (in our case $X_s=(W_s,B_s-2\mu s)$ and $q(x_1,x_2)=-e^{2x_2}$). We define $q$-potential operator

 \begin{equation} \label{Gf11}
 V_D^{q}f(x)=E^{x}[\int_0^{\tau_D}e_q(t) f(B_t-2\mu t,W_t)dt]\,.
 \end{equation}

 Conditional $q$-potential operator:

 \begin{equation} \label{Gf12}
 V_D^{q,z}f(x)=E_z^{x}[\int_0^{\tau_D}e_q(t) f(B_t-2\mu t,W_t)dt]\,.
 \end{equation}

 $$V_D^{q,z}(x,w)=h(x,z)^{-1}V_D^{q}(x,w)h(w,z),$$
 where $h(x,z)$ is a Poisson kernel for $D$ of $(B_t-2\mu t,W_t)$. The formula for it isgiven in the next section.

 Now compute $$
 V_D^{q,z}q(x)=E_z^{x}[\int_0^{\tau_D}e_q(t) q(X_t)dt]=E_z^{x}[e_q(t)|^{\tau_D}_0] =E_z^{x}e_q(\tau_D)-1$$

 The last equality follows from the fact that $\frac {d e_q}{dt}=e_q(t) q(X_t)$.

 We state an important relation between the $\lambda$-Green function of the
 set $D$ and the conditional gauge $E_z^{(0,y)} e^{-\lambda \tau_D}$. Although
 it is well-known (\cite{ChZ}), we provide the proof for the sake of convenience.

 \begin{lem} \label{gauge1}
 \begin{equation} \label{g01}
 \int_D  V_D^{q}(x,w)q(w) h(w,z) dw =
 h(x,z) \(1-E_z^{x} e_q( \tau_D)\)
  \end{equation}
  \end{lem}
  \begin{proof}
  Let $\Phi$ be a bounded measurable function of real variable. By the strong
  Markov property we obtain
  \begin{eqnarray*}
  &{}&\int_{-\infty}^0\int_D  V_D^{q}(x,w)q(w) h(w,z)\Phi(z) dw
    dz \\
   & =&
   E^{x}[\int_0^{\tau_D}q(X_t)e_q(t) E^{X_t}\Phi(X_{\tau_D})dt]\\
  &=&  E^{x}[\int_0^{\tau_D}q(X_t)e_q(t) \Phi(X_{\tau_D})\circ \theta_t dt]
 = E^{x}[\Phi(X_{\tau_D})\int_0^{\tau_D}q(X_t)e_q(t) ]dt \\
  &=&  E^{x}[\Phi(X_{\tau_D})\(e_q(\tau_D)-1\)]=
   E^{x}[e_q(\tau_D)\Phi(X_{\tau_D})]-E^{x}[\Phi(X_{\tau_D})] \\
  &=& E^{x}[E_{X_{\tau_D}}^{x}[ e_q( \tau_D)]\Phi(X_{\tau_D})]-E^{x}[\Phi(X_{\tau_D})] \,.
   \end{eqnarray*}
  The result is now a direct consequence of comparison of the densities functions.
  \end{proof}
 \begin{lem} \label{gauge}
 \begin{equation} \label{g0}
 \lambda \int_D  G_D^{\lambda}((y,0),(x_1,x_2)) h((x_1,x_2),z) dx_1 dx_2 =
 h(y,z) \(1-E_z^{(y,0)} e^{-\lambda \tau_D}\)\,.
  \end{equation}
  \end{lem}
  \begin{proof}
  Let $\Phi$ be a bounded measurable function of real variable. By the strong
  Markov property we obtain
  \begin{eqnarray*}
  &{}& \lambda \int_{-\infty}^0 \int_D  G_D^{\lambda}((y,0),(x_1,x_2)) h((x_1,x_2),z)
   \Phi(z) dz dx_1 dx_2 \\
   & =&
  \lambda E^{(0,0)}[\int_0^{\tau_D}e^{-\lambda t} E^{(B_t,W_t)}\Phi(B_{\tau_D})dt]\\
  &=& \lambda E^{(y,0)}[\int_0^{\tau_D}e^{-\lambda t} \Phi(B_{\tau_D})\circ \theta_t dt]
 = E^{(y,0)}[\Phi(B_{\tau_D})\int_0^{\tau_D}\lambda e^{-\lambda t} ]dt \\
  &=&  E^{(y,0)}[\Phi(B_{\tau_D})\(1-e^{-\lambda \tau_D} \)]=
   E^{(y,0)}[\Phi(B_{\tau_D})]- E^{(y,0)}[ e^{-\lambda \tau_D}\Phi(B_{\tau_D})] \\
  &=& E^{(y,0)}[\Phi(B_{\tau_D})]-
   E^{(y,0)}[E_{B_{\tau_D}}^{(y,0)}[ e^{-\lambda \tau_D}]\Phi(B_{\tau_D})] \,.
   \end{eqnarray*}
  The result is now a direct consequence of comparison of the densities functions.
  \end{proof}
 We now need another formula of the sweeping-out type
  pertaining the $\lambda$-Green function (of the set $D$):
 \begin{equation} \label{Gf2}
 G_D^{\lambda}(u,v)= U^{\lambda}(u,v)-E^u[e^{-\lambda \tau_D}
  U^{\lambda}((B_{\tau_D},W_{\tau_D}),v)]\,.
 \end{equation}
 We begin with the following calculation:
 \begin{eqnarray*}
  &{}& E^{(y,0)}[\int_{D}e^{-\lambda \tau_D}
  U^{\lambda}((B_{\tau_D},W_{\tau_D}),(x_1,x_2)) h((x_1,x_2),z)\, dx_1 dx_2] \\
  &=& \int_{-\infty}^0 \int_0^{\infty} \int_{D}e^{-\lambda s}
   U^{\lambda}((v,0),(x_1,x_2))h_y(v,v) h((x_1,x_2),z)\,dv\, ds\, dx_1 dx_2 \\
   &=& \int_{-\infty}^0 \int_0^{\infty} \int_{D}e^{-\lambda s}
     U^{\lambda}((y,0),(x_1,x_2))h_y(s,v) h((x_1,x_2 -(y-v)),z)\,dv\, ds\, dx_1 dx_2\,.
  \end{eqnarray*}
 We need to compute now
  \begin{equation} \label{basic}
  \int_{-\infty}^0 \int_0^{\infty} h_y(v,s) h((x_1,x_2 -(y-v)),z)dv ds =
   \Psi(y-z,x_1,x_2)\,.
   \end{equation}
  Observe that we have
  \begin{equation*}
  \int_D  U^{\lambda}((y,0),(x_1,x_2)) \(  h((x_1,x_2),z)-\Psi(y-z,x_1,x_2)\)\,dx_1 dx_2
=  h(y,z)\(1-E^{y-z} e^{-\lambda \sigma}\)\,.
   \end{equation*}
  On the other hand, we have the following formula for the $\lambda$-Green function
  of the set $\H$:

  \begin{equation*}
  G_{\H}^{\lambda}((y-z,0),(x_1,x_2))= U^{\lambda}((y-z,0),(x_1,x_2))-
   U^{\lambda}((y-z,0),(-|x_1|,x_2)) \,.
  \end{equation*}
 Here $U^{\lambda}$ is the
      $\lambda$-potential of the process $(W_t,B_t)$. We have
      \begin{equation} \label{pot}
      U^{\lambda}(u,v)=U^{\lambda}(|u-v|)={1 \o \sqrt{2} \pi} {\sqrt{\lambda} \o |u-v|}
      K_1(\sqrt{\lambda}|u-v|)\,.
      \end{equation}
  A direct computation yields
  \begin{equation*}
 \lambda \int_{{\R}^2}  G_{\H}^{\lambda}((y-z,0),(x_1,x_2)) \,dx_1 dx_2=
    1-e^{-\sqrt{\lambda}(y-z)} \,.
    \end{equation*}

\section{Transition density of the conditional process}\label{td}

In this section we give a  representation formula for the
transition density of the process $(B(t),W(t))$ conditioned to
exit $D$ at the specified point $(z,0)$ of $D^{c}$.

The following technical lemma is essential in what follows.
We recall that $\tau$ denotes the first hitting time of the point $0$
by the process $B(t)$. By $p(t,x,y)$ we denote the transition density
of the process $(B(t),W(t))$ and $ p_z^{D}(t,(y,0),w)$ denotes the
transition density of the process conditioned to exit $D$ at the boundary
point $\partial D$: $(B_{\tau_D},W_{\tau_D})=(B_{\tau_D},0)=(z,0)$.

\begin{lem}\label{cond}
Define
 \begin{equation} \label{rzD}
 r_z^{D}(t,(y,0),w)=E^{y-z}[\tau<t;p(t-\tau,(z,0),w)]\,.
 \end{equation}
Then we have
 \begin{equation} \label{pzD}
  p_z^{D}(t,(y,0),w)= h(y,z)^{-1}p(t,(y,0),w)h(w,z)-r_z^{D}(t,(y,0),w)\,.
  \end{equation}
  \end{lem}
\begin{proof}
 By the strong Markov property we obtain
 \begin{equation*}
 E^{(y,0)}[f(B_{t+\tau_D},W_{t+\tau_D})|{\cal{F}}_{\tau_D}]=
 E^{(B_{\tau_D},0)}[f(B_t,W_t)]\,.
  \end{equation*}
 By approximation, for a positive random variable $l$ which is
 ${\cal{F}}_{\tau_D}$-measurable, we obtain for $Y_t=(B_{t+\tau_D},W_{t+\tau_D})$
 \begin{equation*}
 E^{(y,0)}[f(Y_l)|{\cal{F}}_{\tau_D}]=
 E^{(B_{\tau_D},0)}[f(B_s,W_s)]|_{s=l}\,.
  \end{equation*}
 Let now $l=t-\tau_D$ on the set $\{\tau_D<t\}$. We then obtain
  \begin{eqnarray*}
  E^{(y,0)}[\tau_D<t;f(Y_t)|{\cal{F}}_{\tau_D}] &=&
  E^{(y,0)}[f(Y_l)|{\cal{F}}_{\tau_D}]\text{\bf{1}}_{\{\tau_D<t\}}= \\
  E^{(y,0)}[\tau_D<t;f(Y_l)|{\cal{F}}_{\tau_D}] &=&
  \text{\bf{1}}_{\{\tau_D<t\}} E^{(B_{\tau_D},0)}[f(B_s,W_s)]|_{s=l}\,.
   \end{eqnarray*}
  Taking the conditional expectation with respect to the random variable
  $B_{\tau_D}$, and taking into account that
  \begin{equation*}
  E^{(y,0)}[\Phi(\tau_D)|B_{\tau_D}=z]= E^{y-z}[\Phi(\tau)]\,,
  \end{equation*}
  we obtain
  \begin{eqnarray*}
    E^{(y,0)}[\tau_D<t;f(B_t,W_t)|B_{\tau_D}] &=&
    E^{(y,0)}[\tau_D<t;E^{(B_{\tau_D},0)}[f(B_s,W_s)]|_{s=l}|B_{\tau_D}] \\
   &=& E^{y-z}[\tau <t;E^{(z,0)}[f(B_s,W_s)]|_{s=t-\tau}]|_{z=B_{\tau_D}}\,.
  \end{eqnarray*}
 This gives at once
 \begin{equation} \label{crz}
 r_z^{D}(t,(y,0),w)=E^{y-z}[\tau<t;p(t-\tau,(z,0),w)]\,.
 \end{equation}

\end{proof}
Under the same notational convention as above we obtain

\begin{lem}\label{killed}
 \begin{eqnarray} \label{kD}
 r^{D}(t,(0,y),w)&=&E^{(0,y)}[\tau_D<t;p(t-\tau_D,(0,B_{\tau_D}),w)]\,,\\
 p^{D}(t,(0,y),w) &=& p(t,(0,y),w)-  r^{D}(t,(0,y),w)\,.
 \end{eqnarray}
 \end{lem}
\begin{proof}
By previous arguments we obtain
\begin{eqnarray*}
E^{(y,0)}[E^{y-z}[\tau<t;p(t-\tau,(z,0),w)]|_{z=B_{\tau_D}}]]&=&
E^{(y,0)}[E^{(y,0)}[\tau_D<t;p(t-\tau_D,(B_{\tau_D},0),w)]|B_{\tau_D}] \\
&=&E^{(y,0)}[\tau_D<t;p(t-\tau_D,(B_{\tau_D},0),w)]\,.
\end{eqnarray*}
This obviously ends the proof.
 \end{proof}
For further applications we write the density of the killed process in a more
computational form:
\begin{cor} \label{pkilled}
\begin{eqnarray*}
 p^{D}(t,(y,0),w)&=& p(t,(y,0),w)- \int_0^t \int_{-\infty}^0 h_y(v,s) p(t-s,(v,0),w)\,dv\,ds \\
 &=&{1 \o \pi} \int_0^t \int_{-\infty}^0 [p(t,(y,0),w)-p(u,(v,0),w)]
  \left({y \o -v} \right)^{1/2}
 {(u-t)^{-3/2} \o \sqrt{2 \pi}}e^{-{(y-v)^2 \o 2(u-t)}} \,du\,dv\,.
\end{eqnarray*}
\end{cor}


 \section{Hiperbolic Cauchy motion}


{\bf Joint density of $(\tau_D,B^\mu_{\tau_D})$}

Let $B^\mu_t=B_t-2\mu t$. Suppose that $(W(t),B(t))$ is starting from the point$(0,y)$. Then under the new probability $Q$ the process
$(W(t),B(t)+2\mu t)$ is a BM(with variance $2t$) starting from the point$(0,y)$. The probability $Q$ satisfies:

$$Z_t=\frac {dQ}{dP}|_{{\cal F}_t}=e^{-\mu B_t-\mu^2t}e^{\mu y}.$$

Now $B_t=  B^{-\mu}_t-2\mu t$ is a BM with drift under $Q$. We can calculate the joint density of $(\tau_D,B{\tau_D})$ under $Q$. First oberve that $\{\tau_D\le t,B_{\tau_D}\le z\}\in {\cal F}_{t\wedge \tau_D }$. Hence

\begin{eqnarray*}
Q(\tau_D\le t,B_{\tau_D}\le z)&=& E\{\tau_D\le t,B_{\tau_D}\le z\}Z_t\\
                            &=& E\{\tau_D\le t,B_{\tau_D}\le z\}E[Z_t|{\cal F}_{t\wedge \tau_D }]\\
														&=& E\{\tau_D\le t,B_{\tau_D}\le z\}Z_{t\wedge \tau_D }\\
														&=& E\{\tau_D\le t,B_{\tau_D}\le z\}Z_{\tau_D} \\
									&=& E\{\tau_D\le t,B_{\tau_D}\le z\}e^{-\mu B_{\tau_D}-\mu^2\tau_D}e^{\mu y} \\
&=&\int_0^t\int_{-\infty}^z  h(y,v) g_{(y-v)}(s)e^{\mu (y-v)}e^{-\mu^2 s}dvdt
\end{eqnarray*}

Let $h^\mu_y(z,s)$ is the joint density of $(\tau_D,B^\mu_{\tau_D})$ of is contained in the following

  \begin{cor} \label{jdensity}
  We have for $z<0<y$ and $s>0$
  \begin{equation} \label{bttd10}
  h^\mu_y(z,s)= h(y,z) g_{(y-z)}(s)e^{\mu (y-z)}e^{-\mu^2 s} \,.
  \end{equation}
  \end{cor}
 We recall that $ g_{a}(s)$ denotes the density of $\tau$ with respect to $P^a$. Also observe that $g^\mu_a(s)=g_{a}(s)e^{\mu a}e^{-\mu^2 s}$ is the density of $\tau^\mu$- hitting time of $0$ by the process $B(t)-\mu t$ with respect to $P^a$.

 Let $h^\mu_y(z)$ be the density of $B^\mu_{\tau_D}$. Integrating $h^\mu_y(z,s)$ with respect to $t$ we obtain a surprising fact that $h^\mu_y(z)=h_y(z)$ - it does not depend on $\mu$!

Let $H_a=R\times R^+\setminus\{0\}\times[0,a]$, $a>0$. Let
$X_{t}=y\exp (B(t)-2\mu t)$ - vertical coordinate of hyperbolic BM. We can easily find the
 joint density of $(X_{\tau_{H_a}}, \tau_{H_a})$ as a consequence of Corollary \ref{jdensity}.

 \begin{cor} \label{jdensityhip}
  We have for $0<z<a<y$ and $s>0$
  \begin{equation} \label{bttd10}
  H^\mu_y(z,s)= z^{-1}h(\ln (y/a),\ln (z/a)) g^{2\mu}_{\ln (y/z)}(s)\,.
  \end{equation}
  \end{cor}

 $$ H^\mu_y(z,s)={1 \o \pi z} \left( \ln (y/a) \o \ln (a/z) \right)^{1/2} {1 \o \ln (y/z)}  \left({y\o z}\right)^{2\mu }g_{\ln (y/z)}(s)e^{-4\mu^2 s}
 $$

Now we can write the density of $X_{\tau_{H_a}}$:
 $$ H^\mu_y(z)={1 \o \pi z} \left( \ln (y/a) \o \ln (a/z) \right)^{1/2} {1 \o \ln (y/z)}  $$

  Define
 \begin{equation*}
 A_y(t)=y^2\int_0^t \exp 2(B(s)-2\mu s)\/ds
\quad \text{and} \quad
 A_{y}(\tau_D)=\int_0^{\tau_D} x^2 \exp 2(B(t)-2\mu t)\/dt\/.
 \end{equation*}
 By the strong Markov property of Brownian motion we obtain

  \textbf{Basic relationship} (for $\mu >0$).
  \newline Observe that $y^2\exp 2(B(\tau_D)-2\mu \tau_D)=X_{\tau_D}^2$, $P^{(0,y)}-a.s$,
   where
  $ X_{\tau_D}$ is the hitting distribution of one-dimensional hiperbolic Cauchy
     motion starting from $y>0$. Hence, we obtain

   \begin{eqnarray}
   A_y(\infty)&=&y^2\int_0^{\tau_D}\exp 2(B(s)-2\mu s)\/ds +
   y^2\int_{\tau_D}^\infty\exp 2(B(s)-2\mu s)\/ds \nonumber \\
   &=&  A_{y}(\tau_D)+ y^2\exp 2(B(\tau_D)-2\mu \tau_D)
   \int_0^\infty \exp 2(B(s+\tau_D)-B(\tau_D)-2\mu s)\/ds\nonumber\\
  & =&A_{y}(\tau_D)+y^{-2} X_{\tau_D}^2  A'_y(\infty)\label{basic},
   \end{eqnarray}
   where $A'_y(\infty)$ is a copy of $A_y(\infty)$, independent from $A_{y}(\tau_D)$
   and $ X_{\tau_D}$.

 The relation \pref{basic} can also be read as follows:
  \begin{equation}\label{basic1}
    A_y(\infty)  = A_{y}(\tau_D)+ {X_{\tau_D}^2 \o y^2} A'_y(\infty)\,.
  \end{equation}
  Taking expectation of both sides, we obtain
 \begin{equation*}
 E^y A(\infty)  =E^{(0,y)} A(\tau_D)+ y^{-2}E^{(0,y)}[X_{\tau_D}^2]E^y A(\infty)\,.
 \end{equation*} If we assume that $ E^y A(\infty)<\infty$ we can compute $=E^{(0,y)} A(\tau_D)$.
Also we can write this for a single $t$:

  \begin{equation}\label{basic2}
    A_y(\infty)  = A_{y}(t)+ {X_{t}^2 \o y^2} A'_y(\infty)\,.
  \end{equation}

  Let $\psi(\lambda)= E\exp(-\lambda A_1(\infty))$. Then the last relationship reads:

  $$ E^y [\exp(-\lambda  A(\infty))|{\cal F}_t]=
 \psi(\lambda X_t^2)\exp(-\lambda A(t)), $$
  which means that $\psi(\lambda X_t^2)\exp(-\lambda A(t))$
is a regular martingale. ${\cal F}_t$ is generated by both coordinates up to $t$.
Taking conditional expectation with respect to $X_{\tau_D}$ we obtain

$$ E^y [\exp(-\lambda  A(\infty))|X_{\tau_D}]=
 \psi(\lambda X_{\tau_D}^2)E^y[\exp(-\lambda A(\tau_D))|X_{\tau_D}]. $$

 Now suppose that $A(\infty)$ and $X_{\tau_D}$ are {\bf independent}
 (maybe wishful thinking) then the left hand side is $\psi(\lambda y^2)$ and we would have:

 \begin{equation}\label{con}
 E^y[\exp(-\lambda A(\tau_D))|X_{\tau_D}]= {\psi(\lambda y^2)\o \psi(\lambda X_{\tau_D}^2)}
 \end{equation}

 {\bf CONJECTURE}

 $A(\infty)$ and $X_{\tau_D}$ are {\bf independent}

 This conjecture is obviously equivalent to (\ref{con})

By the same argument from (\ref{basic2}) it follows that

$$A_{y}(t)+ X_{t}^2  c$$
is a regular martingale provided  $E A_1(\infty))= c $ is finite.


\end{document}